\documentclass{amsart}
\usepackage{amsmath,amsfonts,amsthm, amssymb,color}
\usepackage[colorlinks,urlcolor=blue,linkcolor=blue,citecolor=blue]{hyperref}
\usepackage[final]{optional} % comm/final for comm/no comm
\newcommand\comm[1]{\opt{comm}
{{\color{red}$\blacktriangleright$\ \small\sf#1$\blacktriangleleft$}}}
\newcommand\francois[1]{{\color{magenta}#1}}
\author{Fran\c cois Digne and Jean Michel}
\address[F.~Digne]{Laboratoire Ami\'enois de Math\'ematique Fondamentale et
Appliqu\'ee, CNRS UMR 7352, Universit\'e de Picardie-Jules Verne,
80039 Amiens Cedex France.}
\email{digne@u-picardie.fr}
\urladdr{www.lamfa.u-picardie.fr/digne}
\address[J.~Michel]{Institut Math\'ematique de Jussieu -- Paris rive
gauche, CNRS UMR 7586, Universit\'e de Paris, B\^atiment Sophie Germain,
75013, Paris France.}
\email{jean.michel@imj-prg.fr}
\urladdr{webusers.imj-prg.fr/$\sim$jean.michel}

\title{Formulae for two-variable Green functions}
\newcommand\bG{{\mathbf G}}
\newcommand\bL{{\mathbf L}}
\newcommand\bP{{\mathbf P}}
\newcommand\bU{{\mathbf U}}
\newcommand\bX{{\mathbf X}}

\newcommand\CB{{\mathcal B}}
\newcommand\CI{{\mathcal I}}
\newcommand\CS{{\mathcal S}}
\newcommand\CX{{\mathcal X}}
\newcommand\CY{{\mathcal Y}}
\newcommand\CZ{{\mathcal Z}}
\newcommand\BF{{\mathbb F}}
\newcommand\tX{{\tilde\CX}}
\newcommand\tY{{\tilde\CY}}
\newcommand\scal[3]{\langle#1,#2\rangle_{#3}}
\newcommand\GF{{\bG^F}}
\newcommand\LF{{\bL^F}}
\newcommand\Lo{{\bL_\CI}}
\newcommand\Low{{\bL_{\CI,w}}}
\newcommand\LowF{{\bL_\CI^{wF}}}
\newcommand\ZLowF{Z^0(\Lo)^{wF}}
\newcommand\io{{\iota_\CI}}
\newcommand\ud{-}
\newcommand\HFZ{H^1(F,Z)}
\newcommand\inv{^{-1}}

\newcommand\Irr{\mathrm{Irr}}
\newcommand\IG{{\CI_\bG}}
\newcommand\rkss{\text{rkss}}
\newcommand\rk{\text{rk}}
\newcommand\chevie{{\tt Chevie}}

\DeclareMathOperator\Ind{\mathrm{Ind}}
\DeclareMathOperator\Res{\mathrm{Res}}
\DeclareMathOperator\Trace{\mathrm{Trace}}

\DeclareMathOperator\res{\mathrm{res}}
\DeclareMathOperator\codim{\mathrm{codim}}
\newtheorem{proposition}[equation]{Proposition}
\newtheorem{corollary}[equation]{Corollary}
\newtheorem{lemma}[equation]{Lemma}
\newcommand\lexp[2]{\kern\scriptspace\vphantom{#2}^{#1}\kern-\scriptspace#2}
\begin{document}
\begin{abstract}
Based  on results  of Digne-Michel-Lehrer  (2003) we  give two formulae for
two-variable  Green  functions  attached  to  Lusztig  induction  in a finite
reductive  group. We present applications  to explicit computation of these
Green  functions,  to  conjectures  of  Malle  and  Rotilio,  and to scalar
products between Lusztig inductions of Gelfand-Graev characters.
\end{abstract}
\maketitle
Let  $\bG$ be a connected reductive group  with Frobenius root $F$; that is,
some   power $F^\delta$  is   a  Frobenius   endomorphism  attached   to  an
$\BF_{q^\delta}$-structure on $\bG$, where $q^\delta$ is a power of a prime $p$.
Let $\bL$  be an $F$-stable Levi subgroup of a (non-necessarily $F$-stable)
parabolic subgroup $\bP$ of $\bG$. Let $\bU$ be the unipotent radical of
$\bP$ and let $\bX_\bU=\{g\bU\in\bG/\bU\mid g\inv\lexp Fg\in\bU\cdot\lexp
F\bU\}$ be the variety used to define the Lusztig induction and restriction
functors $R_\bL^\bG$ and $\lexp *R_\bL^\bG$.
For $u\in\GF, v\in\LF$ unipotent elements, the two-variable Green function is
defined as $$Q^\bG_\bL(u,v)=\Trace((u,v)\mid\sum_i(-1)^i H^i_c(\bX_\bU)).$$

In  this  paper,  using  the  results of \cite{DLM3}, we give two different
formulae  for two-variable Green functions, and some consequences of these,
including proving some conjectures of \cite{MR}.

The two-variables Green functions occur in the character formulae for Lusztig
induction and restriction. In particular, for unipotent elements these
formulae read
\begin{proposition}(See for example \cite[10.1.2]{DMbook})\label{unipotent}
\begin{itemize}
\item
If $u$ is a unipotent element of $\GF$, and $\psi$ a class function on $\LF$,
we have $$R_\bL^\bG(\psi)(u)=|\LF|\scal\psi{Q_\bL^\bG(u,\ud)}\LF.$$
\item
If $v$ is a unipotent element of $\LF$, and $\chi$ a class function on $\GF$,
we have $$\lexp *R_\bL^\bG(\chi)(v)=|\LF|\scal\chi{Q_\bL^\bG(\ud,v\inv)}\GF.$$
\end{itemize}
\end{proposition}

\section*{Two formulae for two-variable Green functions}
For an element $u$ in a group $G$ we denote by $u^G$ the $G$-conjugacy class
of $u$.
\begin{proposition}\label{QLG=} Assume either the centre $Z\bG$ of $\bG$ is
connected and $q>2$, or $q$ is large enough (depending just on the Dynkin
diagram of $\bG$). Then for $u$ regular, $Q_\bL^\bG(u,\ud)$ vanishes outside
a unique regular unipotent class of $\LF$. For $v$ in that class,
we have $Q_\bL^\bG(u,v)=|v^\LF|\inv$.
\end{proposition}
\begin{proof}[Proof (Rotilio)]
Let $\gamma_u^\bG$ be the normalized characteristic function of the
$\GF$-conjugacy class of $u$; that is, the function equal to $0$ outside the
class of $u$ and to $|C_\bG(u)^F|$ on that class.
For $v'\in\LF$ unipotent, Proposition \ref{unipotent} gives
 $$\lexp *R^\bG_\bL(\gamma_u^\bG)(v')
 =|\LF|\scal{\gamma_u^\bG}{Q_\bL^\bG(\ud,v^{\prime-1})}\GF
 =|\LF|Q_\bL^\bG(u,v^{\prime-1}).$$
Now, by \cite[Theorem 15.2]{Bonnafe}, there exists $v\in\LF$ such that the
left-hand side is equal to $\gamma_v^\bL(v')$. In \cite[Theorem 15.2]{Bonnafe}
$q$ is assumed large enough so that the Mackey formula holds and that $p$
and $q$ are such that the results of \cite{GF} hold (that is $p$ almost good
and $q$ larger than some number depending on the Dynkin diagram of $\bG$).
The Mackey formula is known to hold if $q>2$ by
\cite{BM} and the results of \cite{GF} hold
without condition on $q$ if $Z\bG$ is connected by \cite{shoji} or \cite[Theorem 4.2]{shoji96}
(the assumption that $p$ is almost good can be removed, see
\cite[section 89]{comments} and \cite{clean}).
The proposition follows.
\end{proof}
\comm{
\begin{proof}
\francois{definir $\hat \gamma$}
Formula (B) in the proof of \cite[Proposition 15.2]{Bonnafe} can be written
$$R_\LowF^\LF(\hat\gamma^\LowF_{v_w,\zeta})(u_{\bL,z_\bL})=
\zeta(\varphi_{\Lo,v}^{\bG}(w))\scal{\hat\gamma^\LowF_{v_w,\zeta}}
{\hat\gamma^\LowF_{v_w,\zeta}}{\LowF},$$
where $z_\bL\in H^1(F,Z(\bL))$ depends on the $\GF$-class of $u$ and
$u_{\bL,z_\bL}$ is a regular unipotent element of $\LF$ whose class is
parametrised by $z_\bL$ with respect to the class of $\res^\bG_\bL(u)$
(notation of \cite{Bonnafe}).

Comparing with the analogous formula with $\bG$ instead of $\bL$, we get
\begin{equation}\label{*}
R_\LowF^\LF(\hat\gamma^{\LowF}_{v_w,\zeta})(u_{\bL,z_\bL})=
R_\LowF^\GF(\hat\gamma^\LowF_{v_w,\zeta})(u).
\end{equation}

By formula \cite[15.8]{Bonnafe} we have
$\hat\gamma^{\LowF}_{v_w,\zeta}=\frac{|(Z(\Low))^{0F}|}
{\zeta(\varphi_{\bL_\CI,v}^{\bG}(w))}
\CX_\io$.
By \cite[Proposition 3.2]{DLM3},
we have $Q^{\bG,\IG}_{wF}=R_\LowF^\GF(\tilde\CX_\io)=
q^{c_{\iota_\bG}}R_\LowF^\GF(\CX_\io)$, where
$c_{\iota_\bG}$ is as in \cite[above Remark 2.1]{DLM3} and
$\iota_\bG$ is the  only local  system with regular support
in $\IG$ (see \cite[1.10]{DLM2}). Using the same formula in $\bL$, equality
(\ref{*})  gives $Q^{\bG,\IG}_{wF}(u)=Q^{\bL,\CI}_{wF}(u_{\bL,z_\bL})$,
using that $c_{\iota_\bG}=c_\iota$
where $\iota$ is the analogue of $\iota_\bG$ for $\bL$.

Thus the term relative to the block $\CI$ in the formula
 of Proposition \ref{somme} can be written
 $|\LF|\inv\sum_{w\in W_\bL{\Lo}}\frac{|\ZLowF|}{|W_\bL(\Lo)|}
 Q^{\bL,\CI}_{wF}(u_{\bL,z_\bL})Q_{wF}^{\bL,\CI}(v)$
which, according to \cite[5.4]{DLM3} vanishes unless $v$ is regular, and
for $v$ regular is equal to
$$|\LF|\inv|H^1(F,Z(\bL))|\inv\sum_{z\in H^1(F,Z(\bL))}|C^0_\bL(u_{\bL,z})^F|
\overline{\CY_\iota(u_{\bL,z})}\CY_\iota(u_{\bL,z_\bL})\CY_\iota(u_{\bL,z})
\overline{\CY_\iota(v)}.$$

The cardinality $|C^0_\bL(u_{\bL,z})^F|$ is independent of $z$ (hence equal
to $|C^0_\bL(u_{\bL,z_\bL})^F|$) since it is $|(Z^0(\bL))^F|$ times $|\bU_z^F|$
where $\bU_z$ is a connected unipotent group whose dimension is independent
of $z$ (see \cite[Lemma 12.2.3]{DMbook}), whence the assertion by,
for instance \cite[4.1.12]{DMbook}. By  \cite[(4.2)]{DLM3}
the above sum becomes
$$|\LF|\inv|C^0_\bL(u_{\bL,z_\bL})^F|
\CY_\iota(u_{\bL,z_\bL}) \overline{\CY_\iota(v)}.$$
By \cite[(4.5)]{DLM3}, the sum over all blocks of this formula
vanishes unless $v$ is $\LF$-conjugate to  $u_{\bL,z_\bL}$
and is equal to $|\LF|\inv |C_\bL(u_{\bL,z_\bL})^F|$ in this last case,
whence the proposition.
\end{proof}}

As in \cite{DLM3}, we consider irreducible $\bG$-equivariant local systems
on  unipotent  classes.  These  local systems are partitioned into
``blocks''  parametrised by  cuspidal pairs formed by a Levi subgroup
and a cuspidal local system supported on a unipotent class of that Levi
subgroup.
Let us call $Q^{\bL,\CI}_{wF}$ the function defined in \cite[3.1(iii)]{DLM3}
relative to an $F$-stable block $\CI$ of unipotently supported local systems
on $\bL$ and to $wF\in W_\bL(\Lo)F$
where $(\Lo,\io)$ is the cuspidal datum of $\CI$ (for a Levi $\bL$ of a
reductive group $\bG$, we set $W_\bG(\bL)=N_\bG(\bL)/\bL$).

\begin{proposition} \label{somme}
Assume either $Z\bG$ is connected or $q$ is large enough 
(depending just on the Dynkin diagram of $\bG$). 
For $u$ a unipotent element of $\GF$ and $v$ a unipotent
element of $\LF$, we have
$$Q^\bG_\bL(u,v)=|\LF|\inv \sum_\CI\sum_{w\in W_\bL(\Lo)}
\frac{|\ZLowF|}{|W_\bL(\Lo)|}\overline {Q^{\bG,\IG}_{wF}(u)}
Q_{wF}^{\bL,\CI}(v),$$
where $\CI$ runs over the $F$-stable blocks of $\bL$ and where
$\CI_\bG$ is the block of $\bG$ with same cuspidal data as $\CI$.
\end{proposition}
The part of the above sum for $\CI$  the principal block is the same formula
as \cite[Corollaire 4.4]{DM}.
\begin{proof}
Proposition  \ref{unipotent} applied with $\psi=Q^{\bL,\CI}_{wF}$
gives, if we write $R_\LF^\GF$ instead of $R_\bL^\bG$
to keep track of the Frobenius,
$$\scal{Q^{\bL,\CI}_{wF}}{Q^\bG_\bL(u,\ud)}\LF=|\LF|\inv
R_\LF^\GF(Q^{\bL,\CI}_{wF})(u)$$

Now we have by \cite[Proposition 3.2]{DLM3}
$Q^{\bL,\CI}_{wF}=R_\LowF^\LF\tX_{\io,wF}$ where
 $\tX_{\io,wF}$ is $q^{c_\io}$ times the characteristic function of $(\io,
wF)$, a class function on $\LowF$. Here, as in
\cite[above Remark 2.1]{DLM3}, for an irreducible $\bG$-equivariant
local system $\iota$, we denote by $C_\iota$ the unipotent $\bG$-conjugacy
class which is the support of $\iota$, and if
$(\Lo,\io)$ is the cuspidal datum of $\iota$ we set
$c_\iota=\frac 12(\codim C_\iota-\dim Z(\Lo))$.
In \cite[Proposition 3.2]{DLM3} the assumptions on $p$ and $q$ come from
\cite{GF} but by the same considerations than at the end of the proof of
Proposition \ref{QLG=} it is sufficient to assume $Z\bG$ connected or 
$q$ large enough.

By the transitivity of Lusztig induction we get
$$\scal{Q^{\bL,\CI}_{wF}}{Q^\bG_\bL(u,\ud)}\LF=|\LF|\inv R_\LowF^\GF(
\tX_{\io,wF})(u)=|\LF|\inv Q_{wF}^{\bG,\IG}(u),$$
where $\CI_\bG$ is the block of $\bG$ with same cuspidal data as $\CI$.
Using the orthogonality of the Green functions $Q^{\bL,\CI}_{wF}$,
see \cite[Corollary 3.5]{DLM3} (where the assumption is that $p$ is almost
good which comes from \cite{CS}, so can be removed now by \cite{clean})
and the fact they form a basis of
unipotently supported class functions on $\LF$, indexed by the
$W_\bL(\Lo)$-conjugacy classes of $W_\bL(\Lo)F$, we get the proposition.
\end{proof}
Proposition \ref{somme} gives a convenient formula to compute automatically
two-variable Green functions. Table 1 gives an example, computed with the
package \chevie\ (see \cite{chevie}).

We denote by $\CY_\iota$ the characteristic function of the $F$-stable local
system $\iota$, and by $A(u)$ the group of
components of the centralizer of a unipotent element $u$.
\begin{proposition}\label{formula2} 
Assume either $Z\bG$ is connected or $q$ is large enough 
(depending just on the Dynkin diagram of $\bG$). 
Let  $R_{\iota,\gamma}$  be  the  polynomials  which  appear in \cite[Lemma
6.9]{DLM3}. Then
$$Q_\bL^\bG(u,v)=|v^\LF|\inv|A(v)|\inv\sum_\CI\sum_{\iota\in\IG^F,
\gamma\in\CI^F}\overline{\CY_\iota(u)}\CY_\gamma(v)R_{\iota,\gamma}q^{c_\iota-c_\gamma},$$
where $c_\iota=\frac 12(\codim C_\iota-\dim Z(\Lo))$.
\end{proposition}
\begin{proof}
For a block $\CI$ of $\bL$ and $\iota\in\CI^F$,
let $\tilde Q_\iota$ be the function of
\cite[(4.1)]{DLM3}. Then by \cite[(4.4)]{DLM3} applied respectively in $\bG$
 and $\bL$ we have
 $$Q^{\bG,\IG}_{wF}(u)=\sum_{\iota\in\IG^F}\tilde Q_\iota(wF)\tY_\iota(u)
\quad\text{ and }\quad
Q^{\bL,\CI}_{wF}(v)=\sum_{\kappa\in\CI^F}\tilde Q_\kappa(wF)\tY_\kappa(v)$$
 where $\tilde\CY_\iota=q^{c_\iota}\CY_\iota$.
Thus, using the notation $\CZ_\Lo$ as in \cite[3.3]{DLM3} to denote the function
$wF\mapsto|\ZLowF|$ on $W_\bG(\Lo)F$,
the term relative to a block $\CI$ in the formula
of Proposition \ref{somme} can be written
$$|\LF|\inv\scal{\CZ_\Lo\sum_{\kappa\in\CI^F}\tilde Q_\kappa\tY_\kappa(v)}
{\sum_{\iota\in\IG^F}\tY_\iota(u)\Res^{W_\bG(\Lo)F}_{W_\bL(\Lo)F}
\tilde Q_\iota}{W_\bL(\Lo)F}.$$
Applying now \cite[Lemma 6.9]{DLM3} this is equal to
$$|\LF|\inv\scal{\CZ_\Lo\sum_{\kappa\in\CI^F}\tilde Q_\kappa\tY_\kappa(v)}
{\sum_{\iota\in\IG^F,\gamma\in\CI^F}\tY_\iota(u)
R_{\iota,\gamma} \tilde Q_\gamma} {W_\bL(\Lo)F},$$
We use now \cite[Corollary 5.2]{DLM3} which says that,
$\scal{\tilde Q_\gamma}{\CZ_\Lo\tilde Q_\kappa}{W_\bL(\Lo)F}=0$
unless $C_\gamma=C_\kappa$ and in this last case is equal to
$$|A(v)|\inv\sum_{a\in A(v)}|C^0_\bL(v_a)^F|
q^{-2c_\gamma}\CY_\gamma(v_a)\overline{\CY_\kappa(v_a)}$$
Thus the previous sum becomes
$$|\LF|\inv\sum_{\iota\in\IG^F,\gamma\in\CI^F}\overline{\tY_\iota(u)}
 R_{\iota,\gamma} |A(v)|\inv\sum_{a\in A(v)}|C^0_\bL(v_a)^F|
 q^{-2c_\gamma}\CY_\gamma(v_a)\sum_{\kappa\in\CI^F}\overline{\CY_\kappa(v_a)}
\tY_\kappa(v).$$
But by \cite[(4.5)]{DLM3} we have
$\sum_\kappa\overline{\CY_\kappa(v_a)}\tY_\kappa(v)=\begin{cases}
q^{c_\kappa}|A(v)^F|&\text{ if $v_a=v$}\\0&\text{otherwise}\end{cases}$,
where  $\kappa$ runs  over all  local systems.  Thus, summing  over all the
blocks, we get the formula in the statement.
\end{proof}
\begin{corollary}
Assume either $Z\bG$ is connected or $q$ is large enough 
(depending just on the Dynkin diagram of $\bG$). 
Then for any unipotent elements $u\in\GF$ and $v\in\LF$ we have:
\begin{enumerate}
\item $Q_\bL^\bG(u,v)$ vanishes unless
$v^\bG\subseteq \overline{u^\bG}\subseteq{\Ind_\bL^\bG(v^\bL)}$,
where $\Ind_\bL^\bG(v^\bL)$ is the induced class in the sense of \cite{LS}.
\item $|v^\LF| |A(v)| Q_\bL^\bG(u,v)$ is an integer and is a polynomial in
$q$ with integral coefficients.
\end{enumerate}
\end{corollary}
\begin{proof}
For (i), we use \cite[Lemma 6.9(i)]{DLM3} which states that
$R_{\iota,\gamma}=0$ unless $C_\gamma\subseteq\overline{C_\iota}\subseteq
\overline{\Ind_\bL^\bG(C_\gamma)}$.
Since $\tY_\kappa(v)$  vanishes unless $C_\kappa\owns
v$,  the only non-zero  terms in the  formula of proposition \ref{formula2}
have  $C_\gamma\owns v$,  whence the  result since  $\tY_\iota(u)$ vanishes
unless $C_\iota\owns u$.

For (ii), we start with
\begin{lemma}
$q^{c_\iota-c_\gamma}R_{\iota,\gamma}$ is a polynomial in $q$ with integral
coefficients.
\end{lemma}
\begin{proof}
The defining equation of the matrix
$\tilde R=\{q^{c_\iota-c_\gamma}R_{\iota,\gamma}\}_{\iota,\gamma}$
reads (see the proof of \cite[Lemma 6.9(i)]{DLM3}):
$$\tilde R=P_\bG C_\bG I C_\bL\inv P_\bL\inv$$
where $C_\bG$ is the diagonal matrix with diagonal coefficients
$q^{c_\iota}$ for $\iota\in\IG$, and $C_\bL$ is the similar matrix for $\bL$ and
$\CI$, where $P_\bG$ is the matrix with coefficients
$\{P_{\iota,\iota'}\}_{\iota,\iota'\in\IG}$ where these polynomials are those
defined in \cite[6.5]{usupp}, and $P_\bL$ is the similar matrix for $\bL$ and
$\CI$, and finally $I$ is the matrix with coefficients
$$I_{\iota,\gamma}=\scal{\Ind^{W_\bG(\Lo)F}_{W_\bL(\Lo)F}
\tilde\varphi_\gamma}{\tilde\varphi_\iota}{W_\bG(\Lo)F}$$
where $\tilde\varphi_\gamma$ is the character of $W_\bL(\Lo)F$ which
corresponds by the generalised Springer correspondence to $\gamma$
(and similarly for $\tilde\varphi_\iota$). Since $P_\bL$ and $P_\bG$ are
unitriangular matrices with coefficients integral polynomials in $q$, thus
$P_\bL\inv$ also, it suffices to prove that  $C_\bG I C_\bL\inv$ has
coefficients polynomial in $q$, or equivalently that
$$\text{ if }
\scal{\Ind^{W_\bG(\Lo)F}_{W_\bL(\Lo)F}
 \tilde\varphi_\gamma}{\tilde\varphi_\iota}{W_\bG(\Lo)F}\ne 0\text{, then }
c_\iota-c_\gamma\ge 0.$$
We now use \cite[Proposition 2.3(ii)]{DLM3} which says that the non-vanishing
above implies $C_\gamma\subseteq\overline{C_\iota}\subseteq
\overline{\Ind_\bL^\bG(C_\gamma)}$. We now use that,
according to the definitions, $c_\iota-c_\gamma=\dim \CB_u^\bG -\dim \CB_v^\bL$
where $\CB_u^\bG$ is the variety of Borel subgroups of $\bG$ containing
an element $u$ of the support of $\iota$, and
where $\CB_v^\bL$ is the variety of Borel subgroups of $\bL$ containing
an element $v$ of the support of $\gamma$. Now the lemma follows from the fact
that by \cite[Theorem 1.3 (b)]{LS} we have $\dim \CB_u^\bG =\dim \CB_v^\bL$ if
$u$ is an
element of $\Ind_\bL^\bG(C_\gamma)$, and that $\dim \CB_u^\bG$ is greater
for  $u\in\overline{\Ind_\bL^\bG(C_\gamma)}-\Ind_\bL^\bG(C_\gamma)$.
\end{proof}
Now (ii) results from the lemma: since the $\tY$ have values algebraic
integers, by Proposition \ref{formula2}
the expression in (ii) is a polynomial in $q$ with coefficients algebraic
integers. But, since
$|\LF| Q_\bL^\bG(u,v)$ is a Lefschetz number (see for example
\cite[8.1.3]{DMbook}), the expression in (ii) is a rational number;
since this is true for an infinite number of integral values of $q$ the
expression in (ii) is a polynomial with integral coefficients.
\end{proof}

\section*{Scalar products of induced Gelfand-Graev characters}
The pretext for this section is as follows:
in   \cite[Remark 3.10]{BM}  is pointed   the  problem  of  computing
$\scal{R^\bG_\bL\Gamma_\iota}{R^\bG_\bL\Gamma_\iota}\GF$  when $(\bG,F)$ is
simply connected of  type $\lexp 2E_6$, when $\bL$ is of  type $A_2\times A_2$, and
when $\iota$ corresponds
to a faithful character of $Z(\bL)/Z^0(\bL)$, and checking that
the value is the same as given by the Mackey formula.
We show now various ways to do this computation, where in this section
we assume $p$ and $q$ large enough for all the results of
\cite{DLM3} to hold (in particular, we assume $p$ good for $\bG$,
thus not solving the problem of loc.\ cit.\ where we need $q=2$).

Let  $Z=Z(\bG)$,  and  let  $\Gamma_z$  be the Gelfand-Graev character
parameterized  by $z\in\HFZ$, see for instance \cite[Definition 2.7]{DLM1}. Let $u_z$ be a representative of the regular
unipotent class parametrized by $z$.
As  in  \cite[7.5 (a)]{usupp}  for  $\iota$ an
$F$-stable  local system  on the  regular unipotent  class
we  define $\Gamma_\iota=c\sum_{z\in\HFZ}\CY_\iota(u_z)\Gamma_z$ where
$c=\frac{|Z/Z^0|}{|\HFZ|}$.

Note that the cardinality $|C_\bG(u_z)^F|$ is independent of $z$; actually it is
equal to $|Z(\bG)^F| q^{\rkss\bG}$ (see \cite[15.5]{Bonnafe}). Thus we will denote
this cardinality $|C_\GF(u)|$ where $u\in\GF$ is any regular unipotent element.
There exists a character $\zeta$ of $\HFZ$ and a root of unity $b_\iota$
(see \cite[above 1.5]{DLM2}) such that $\CY_\iota(u_z)=b_\iota\zeta(z)$.
With these notations, we have

\begin{proposition}\label{Gamma=Y}
We have $\Gamma_\iota=\eta_\bG\sigma_\zeta\inv
c|C_\GF(u)|D\CY_\iota$ where $\eta_\bG$ and $\sigma_\zeta$ are defined as in
\cite[2.5]{DLM2},
\end{proposition}
\begin{proof}
This proposition could be deduced from \cite[Theorem 2.8]{DLM4} using
\cite[Theorem 2.7]{DLM2}. We give here a more elementary proof.

With the notations of \cite[(3.5')]{DLM1} we have
$D\Gamma_z=\sum_{z'\in \HFZ}c_{z,z'}\gamma_{z'}$.
By \cite[lemma 2.3]{DLM2} we have $c_{z,z'}=c_{zz^{\prime-1},1}$ and
$\sum_{z\in\HFZ}\zeta(z)c_{z,1}=\eta_\bG\sigma_\zeta\inv$. It follows that
$$\begin{aligned}
c\inv b_\iota\inv D\Gamma_\iota&=\sum_{z\in \HFZ}\zeta(z)D\Gamma_z
=\sum_{z,z'\in\HFZ}\zeta(z)c_{z,z'}\gamma_{z'}\\
&=\sum_{z,z'\in\HFZ}c_{zz^{\prime-1},1}\zeta(z)\gamma_{z'}\\
&=\sum_{z'\in\HFZ}\zeta(z')\gamma_{z'}\sum_{z''\in\HFZ}c_{z'',1}\zeta(z'')\\
&=\eta_\bG\sigma_\zeta\inv\sum_{z'\in\HFZ}\zeta(z')\gamma_{z'}
 =\eta_\bG\sigma_\zeta\inv b_\iota\inv|C_\GF(u)|\CY_\iota
\end{aligned}$$
\end{proof}

%We now consider the local system $\iota$ on the regular unipotent class of
%$\bL$ corresponding to the character $\zeta$ of $H^1(F,Z(\bL))$.
\begin{proposition}\label{RLG Gamma}
If $\iota$ is a local
system supported on the regular unipotent class of $\bL$ and $\CI$ denotes
its block, we have
\begin{multline*}
\scal{R_\bL^\bG\Gamma^\bL_\iota}{R_\bL^\bG\Gamma^\bL_{\iota}}\GF=\\
|\frac{Z(\bL)}{Z^0(\bL)}|^2
\sum_{w\in W_\bL(\Lo)}\frac{|\ZLowF||W_\bG(\Lo)|}
{|W_\bL(\Lo)|^2}\frac{|(wF)^{W_\bG(\Lo)}\cap{W_\bL(\Lo)}|}
{|(wF)^{W_\bG(\Lo)}|}.
\end{multline*}
\end{proposition}
Note that in a given block $\CI$ there is at most one local system supported
by the regular unipotent class (see \cite[Corollary 1.10]{DLM2}).
\begin{proof}
When $\iota$ is supported by the regular unipotent class we have $\tilde
Q_\iota=1$, see the begining of section 7, bottom of page 130 in \cite{DLM3}.
Using this in the last formula of the proof of
\cite[Proposition 6.1]{DLM3}, we get that $\Gamma_\iota^\bL$ is up to a root
of unity equal to
$|A(C_\iota))| |W_\bL(\Lo)|\inv\sum_{w\in W_\bL(\Lo)}
|\ZLowF|Q_{wF}^{\bL,\CI}$.
Since $R_\bL^\bG Q_{wF}^{\bL,\CI}=Q_{wF}^{\bG,\IG}$, we get
\begin{equation*}
\begin{split}
\scal{R_\bL^\bG\Gamma_\iota}{R_\bL^\bG\Gamma_{\iota}}\GF&=\\
|A(C_\iota))|^2|W_\bL&(\Lo)|^{-2}\sum_{w,w'\in W_\bL(\Lo)}
|\ZLowF||Z^0(\Lo)^{w'F}|\scal{Q_{wF}^{\bG,\IG}}{Q_{w'F}^{\bG,\IG}}\GF.
\end{split}
\end{equation*}
By \cite[3.5]{DLM3} the last scalar product is zero unless $wF$ and $w'F$ are
conjugate in $W_\bG(\Lo)$,
and is equal to $|C_{W_\bG(\Lo)}(wF)|/|\ZLowF|$ otherwise.
We get
\begin{equation*}
\begin{split}
\scal{R_\bL^\bG\Gamma_\iota}{R_\bL^\bG\Gamma_\iota}\GF&=
|A(C_\iota))|^2|W_\bL(\Lo)|^{-2}\\&\sum_{w\in W_\bL(\Lo)}
|\ZLowF||C_{W_\bG(\Lo)}(wF)||(wF)^{W_\bG(\Lo)}\cap{W_\bL(\Lo)}|,
\end{split}
\end{equation*}
which gives the formula of the proposition since $A(C_\iota)=Z(\bL)/Z(\bL)^0$.
\end{proof}
\begin{corollary}
Let $\iota$ and $\iota'$ be local
systems supported on the regular unipotent class of $\bG$, and $\CI$, $\CI'$
be their respective blocks: then
\begin{enumerate}
\item
$
\scal{\Gamma^\bG_\iota}{\Gamma^\bG_{\iota'}}\GF=
\begin{cases}
0&\text{if }\iota\neq\iota',\\
|\frac{Z(\bG)}{Z^0(\bG)}|^2 |Z^0(\bG)^F|q^{\dim Z(\Lo)-\dim Z(\bG)}
&\text{if }\iota=\iota'.\end{cases}
$
\item
$\scal{\CY_\iota}{\CY_{\iota'}}\GF=
\begin{cases}
0&\text{if }\iota\neq\iota',\\
q^{-\rkss\bG}|Z^0(\bG)^F|\inv&\text{if }\iota=\iota'.\end{cases}
$
\end{enumerate}
\end{corollary}
\begin{proof}
The functions $Q^{\bG,\IG}_{wF}$ and $Q^{\bG,\CI'_\bG}_{w'F}$
are orthogonal to each other when $\IG\neq\CI'_\bG$ (see \cite[V, 24.3.6]{CS}
where the orthogonality is stated for the functions $\CX_\iota$).
Since there is a unique $\iota$ in a given block supported on the regular
unipotent class, we get the orthogonality in (i).
In the case $\iota=\iota'$ in (i), the specialization $\bL=\bG$ in Proposition
\ref{RLG Gamma} is
$\scal{\Gamma^\bG_\iota}{\Gamma^\bG_{\iota'}}\GF=
|\frac{Z(\bG)}{Z^0(\bG)}|^2
\sum_{w\in W_\bG(\Lo)}\frac{|\ZLowF|}{|W_\bG(\Lo)|}$.
By \cite[Corollary 5.2]{DLM3}, where we use that $\tilde Q_\iota=1$
when $\iota$ has regular support, we have
$\sum_{w\in W_\bG(\Lo)}\frac{|\ZLowF|}
{|W_\bG(\Lo)|}=q^{-2c_\iota}|C_\bG(u)^{0F}|$.
Whence
$\scal{\Gamma^\bG_\iota}{\Gamma^\bG_{\iota'}}\GF=
|\frac{Z(\bG)}{Z^0(\bG)}|^2q^{-\rk\bG+\dim Z(\Lo)}|C^0_\bG(u)^F|$.
Using $|C^0_\bG(u)^F|=q^{\rkss\bG}|Z^0(\bG)^F|$, we get (i).

For (ii), we  apply Proposition \ref{Gamma=Y} in (i), using that
$D$ is an isometry and that
$\sigma_\zeta\overline{\sigma_\zeta}=q^{\rkss\Lo}$
by \cite[proposition 2.5]{DLM2}.
\end{proof}
A particular case of Proposition \ref{RLG Gamma} is
\begin{corollary}\label{BM1}
 If $(\bL,\iota)$ is a cuspidal pair, that is $\bL=\Lo$, then
$$\scal{R_\bL^\bG\Gamma^\bL_\iota}{R_\bL^\bG\Gamma^\bL_\iota}\GF=
|\frac{Z(\bL)}{Z^0(\bL)}|^2 |W_\bG(\bL)||Z(\bL)^{0F}|.$$
\end{corollary}
We remark that this coincides with the value predicted by the Mackey formula
$$\scal{R_\bL^\bG\Gamma^\bL_\iota}{R_\bL^\bG\Gamma^\bL_\iota}\GF=
\sum_{x\in \LF\backslash \CS(\bL,\bL)/\LF}\scal
{\lexp *R^\bL_{\bL\cap\lexp x\bL}(\Gamma^\bL_\iota)}
{\lexp *R^{\lexp x\bL}_{\bL\cap\lexp x\bL}(\lexp x\Gamma^\bL_\iota)}
{(\bL\cap\lexp x\bL)^F} $$

Indeed, since the block $\CI$ which contains the local system $\zeta$ is reduced
to the unique cuspidal local system $(C,\zeta)$ where $C$ is the regular class
of $\bL$,
all terms in the Mackey formula where $\bL\cap\lexp x\bL\ne\bL$ vanish.
Thus the Mackey formula reduces to
$$\scal{R_\bL^\bG\Gamma^\bL_\iota}{R_\bL^\bG\Gamma^\bL_\iota}\GF=
\sum_{x\in W_\bG(\bL)} \scal{\Gamma^\bL_\iota} {\lexp x\Gamma^\bL_\iota} \LF$$
and any $x$ in $W_\bL(\bL)$ acts trivially on $H^1(F,Z(\bL))$ since, the map $h_\bL$
being surjective, any element of $H^1(F,Z(\bL))$ is represented
by an element of $H^1(F,Z(\bG))$; thus all the terms in the sum are equal,
and we get the same result as Corollary \ref{BM1} by applying Corollary \ref{BM1} in the case $\bG=\bL$.

Another method for computing
$\scal{R_\bL^\bG D\Gamma_i}{R_\bL^\bG D\Gamma_i}\GF$ would be to use Proposition
\ref{unipotent} and the values of the two-variable Green functions. 

We  give these  values in  the following  table in  the particular  case of
$\lexp  2E_6$ for the  $F$-stable standard Levi  subroup of type $A_2\times
A_2$, for $q\equiv -1\pmod 3$, so that $F$ acts trivially on
$Z(\bG)/Z^0(\bG)$. 
This table has been computed in \chevie\ using Proposition \ref{somme}.
The  method is to compute the  one-variable Green functions which appear in
the  right-hand side  sum by  the Lusztig-Shoji  algorithm; note  that even
though  the characteristic function of cuspidal character sheaves are known
only  up to a root of unity,  this ambiguity disappears when doing the sum,
since  such a scalar  appears multiplied by  its complex conjugate. However
the  Lusztig-Shoji  algorithm  depends  also  on  the  knowledge  that when
$\iota$,   with  support  the  class  of  the  unipotent  element  $u$,  is
parameterized   by  $(u,\chi)$  where   $\chi\in\Irr(A(u))$  then  for  the
unipotent  element $u_a\in\GF$ parameterized by  $a\in H^1(F,A(u))$ we have
$\CY_\iota(u_a)=\chi(a)$.  We assume that this hold.
This is known  when $\iota$ is  in the principal
block,  but not  for the  two blocks  with cuspidal datum 
supported on the Levi subgroup of type $A_2\times A_2$.

Note that the table shows that the values of $|v^\LF|Q^\bG_\bL(u,v)$
are not in general polynomials with integral coefficients
but may have denominators equal to $|A(v)|$.
\begin{center}\textsc{ Table 1.} Values of $|v^\LF|Q^\bG_\bL(u,v)$ for
$\bG=\lexp2E_{6}(q)$ simply connected and
$\bL=A _2(q^2)(q-1)^2$, for $q\equiv -1 \pmod 3$.
\end{center}
\begin{footnotesize}
$$
\begin{array}{c|cccccccc}
v\backslash u&E_6&\mbox{$E_6$}_{(\zeta_3)}&\mbox{$E_6$}_{(\zeta_3^2)}&E_6(a_1)&\mbox{$E_6(a_1)$}_{(\zeta_3)}&\mbox{$E_6(a_1)$}_{(\zeta_3^2)}&D_5&E_6(a_3)\\
\hline
111,111&0&0&0&0&0&0&0&0\\
21,21&0&0&0&0&0&0&1&4q+1\\
3,3&1&0&0&(4q+1)/3&\Phi_{2}/3&\Phi_{2}/3&2q\Phi_{2}/3&(7q^{2}+2q-2)q/3\\
\mbox{$3,3$}_{(\zeta_3)}&0&1&0&\Phi_{2}/3&(4q+1)/3&\Phi_{2}/3&2q\Phi_{2}/3&(q-2)q\Phi_{2}/3\\
\mbox{$3,3$}_{(\zeta_3^2)}&0&0&1&\Phi_{2}/3&\Phi_{2}/3&(4q+1)/3&2q\Phi_{2}/3&(q-2)q\Phi_{2}/3\\

\end{array}
$$

$$
\begin{array}{c|ccccc}
v\backslash u&\mbox{$E_6(a_3)$}_{(-\zeta_3^2)}&\mbox{$E_6(a_3)$}_{(\zeta_3)}&\mbox{$E_6(a_3)$}_{(-1)}&\mbox{$E_6(a_3)$}_{(\zeta_3^2)}&\mbox{$E_6(a_3)$}_{(-\zeta_3)}\\
\hline
111,111&0&0&0&0&0\\
21,21&2q+1&4q+1&2q+1&4q+1&2q+1\\
3,3&q^{2}\Phi_{2}&(q-2)q\Phi_{2}/3&(3q+2)q^{2}&(q-2)q\Phi_{2}/3&q^{2}\Phi_{2}\\
\mbox{$3,3$}_{(\zeta_3)}&(3q+2)q^{2}&(q-2)q\Phi_{2}/3&q^{2}\Phi_{2}&(7q^{2}+2q-2)q/3&q^{2}\Phi_{2}\\
\mbox{$3,3$}_{(\zeta_3^2)}&q^{2}\Phi_{2}&(7q^{2}+2q-2)q/3&q^{2}\Phi_{2}&(q-2)q\Phi_{2}/3&(3q+2)q^{2}\\

\end{array}
$$

$$
\begin{array}{c|cccc}
v\backslash u&A_5&\mbox{$A_5$}_{(\zeta_3)}&\mbox{$A_5$}_{(\zeta_3^2)}&D_5(a_1)\\
\hline
111,111&0&0&0&0\\
21,21&(-2q-1)\Phi_{2}&(-2q-1)\Phi_{2}&(-2q-1)\Phi_{2}&3q+1\\
3,3&q\Phi_{2}\Phi_{3}/3&(-5q^{2}-2q+1)q\Phi_{2}/3&(-5q^{2}-2q+1)q\Phi_{2}/3&q\Phi_{1}\Phi_{2}/3\\
\mbox{$3,3$}_{(\zeta_3)}&(-5q^{2}-2q+1)q\Phi_{2}/3&q\Phi_{2}\Phi_{3}/3&(-5q^{2}-2q+1)q\Phi_{2}/3&q\Phi_{1}\Phi_{2}/3\\
\mbox{$3,3$}_{(\zeta_3^2)}&(-5q^{2}-2q+1)q\Phi_{2}/3&(-5q^{2}-2q+1)q\Phi_{2}/3&q\Phi_{2}\Phi_{3}/3&q\Phi_{1}\Phi_{2}/3\\

\end{array}
$$

$$
\begin{array}{c|cccc}
v\backslash u&A_4{+}A_1&D_4&A_4&\mbox{$D_4(a_1)$}_{(111)}\\
\hline
111,111&0&1&0&4q+1\\
21,21&\Phi_{2}\Phi_{3}&3q\Phi_{2}\Phi_{6}&(3q^{3}+q^{2}+q+1)\Phi_{2}&(8q^{3}+2q^{2}+4q-2)q\Phi_{2}\\
3,3&(2q+1)q^{3}\Phi_{2}/3&q\Phi_{1}\Phi_{2}^{2}\Phi_{6}/3&q^{4}\Phi_{2}^{2}&(4q+1)q^{4}\Phi_{2}^{2}/3\\
\mbox{$3,3$}_{(\zeta_3)}&(2q+1)q^{3}\Phi_{2}/3&q\Phi_{1}\Phi_{2}^{2}\Phi_{6}/3&q^{4}\Phi_{2}^{2}&(4q+1)q^{4}\Phi_{2}^{2}/3\\
\mbox{$3,3$}_{(\zeta_3^2)}&(2q+1)q^{3}\Phi_{2}/3&q\Phi_{1}\Phi_{2}^{2}\Phi_{6}/3&q^{4}\Phi_{2}^{2}&(4q+1)q^{4}\Phi_{2}^{2}/3\\

\end{array}
$$

$$
\begin{array}{c|cccc}
v\backslash u&\mbox{$D_4(a_1)$}_{(21)}&D_4(a_1)&A_3{+}A_1&A_3\\
\hline
111,111&2q+1&\Phi_{2}&(-2q-1)\Phi_{2}&(3q^{3}+2q+1)\Phi_{2}\\
21,21&(8q^{3}+6q^{2}+2q+2)q^{2}&(2q+1)q\Phi_{2}\Phi_{6}&(-4q^{3}-q^{2}-2q+1)q\Phi_{2}^{2}&(3q^{3}-q^{2}+2q-1)q\Phi_{2}^{2}\Phi_{4}\\
3,3&(2q+1)q^{4}\Phi_{1}\Phi_{2}/3&q^{4}\Phi_{2}\Phi_{6}/3&(-2q-1)q^{5}\Phi_{2}^{2}/3&0\\
\mbox{$3,3$}_{(\zeta_3)}&(2q+1)q^{4}\Phi_{1}\Phi_{2}/3&q^{4}\Phi_{2}\Phi_{6}/3&(-2q-1)q^{5}\Phi_{2}^{2}/3&0\\
\mbox{$3,3$}_{(\zeta_3^2)}&(2q+1)q^{4}\Phi_{1}\Phi_{2}/3&q^{4}\Phi_{2}\Phi_{6}/3&(-2q-1)q^{5}\Phi_{2}^{2}/3&0\\

\end{array}
$$

$$
\begin{array}{c|cccc}
v\backslash u&2A_2{+}A_1&\mbox{$2A_2{+}A_1$}_{(\zeta_3)}&\mbox{$2A_2{+}A_1$}_{(\zeta_3^2)}&2A_2\\
\hline
111,111&\Phi_{2}\Phi_{3}&\Phi_{2}\Phi_{3}&\Phi_{2}\Phi_{3}&\Phi_{2}^{2}\Phi_{3}\Phi_{6}\\
21,21&(2q^{3}+2q^{2}+4q+1)q^{3}\Phi_{2}&(2q^{3}+2q^{2}+4q+1)q^{3}\Phi_{2}&(2q^{3}+2q^{2}+4q+1)q^{3}\Phi_{2}&3q^{4}\Phi_{2}^{2}\Phi_{3}\Phi_{6}\\
3,3&q^{6}\Phi_{2}\Phi_{3}&0&0&q^{6}\Phi_{2}^{2}\Phi_{3}\Phi_{6}\\
\mbox{$3,3$}_{(\zeta_3)}&0&q^{6}\Phi_{2}\Phi_{3}&0&0\\
\mbox{$3,3$}_{(\zeta_3^2)}&0&0&q^{6}\Phi_{2}\Phi_{3}&0\\

\end{array}
$$

$$
\begin{array}{c|cccc}
v\backslash u&\mbox{$2A_2$}_{(\zeta_3)}&\mbox{$2A_2$}_{(\zeta_3^2)}&A_2{+}2A_1&A_2{+}A_1\\
\hline
111,111&\Phi_{2}^{2}\Phi_{3}\Phi_{6}&\Phi_{2}^{2}\Phi_{3}\Phi_{6}&(2q^{4}+q^{3}+q^{2}+q+1)\Phi_{2}&\Phi_{2}^{2}\Phi_{3}\Phi_{6}\\
21,21&3q^{4}\Phi_{2}^{2}\Phi_{3}\Phi_{6}&3q^{4}\Phi_{2}^{2}\Phi_{3}\Phi_{6}&(q^{3}+2q+1)q^{5}\Phi_{2}&(3q^{2}+2q+1)q^{4}\Phi_{2}^{2}\Phi_{6}\\
3,3&0&0&0&0\\
\mbox{$3,3$}_{(\zeta_3)}&q^{6}\Phi_{2}^{2}\Phi_{3}\Phi_{6}&0&0&0\\
\mbox{$3,3$}_{(\zeta_3^2)}&0&q^{6}\Phi_{2}^{2}\Phi_{3}\Phi_{6}&0&0\\

\end{array}
$$

$$
\begin{array}{c|cc}
v\backslash u&\mbox{$A_2$}_{(11)}&A_2\\
\hline
111,111&(3q^{5}+q^{2}+q+1)\Phi_{2}^{2}\Phi_{6}&5q^{9}+3q^{8}+4q^{7}+4q^{6}+5q^{5}+4q^{4}+2q^{3}+2q^{2}+2q+1\\
21,21&(4q^{2}+q+1)q^{4}\Phi_{2}^{3}\Phi_{6}^{2}&(2q+1)q^{4}\Phi_{1}\Phi_{2}\Phi_{3}\Phi_{4}\Phi_{6}\\
3,3&0&0\\
\mbox{$3,3$}_{(\zeta_3)}&0&0\\
\mbox{$3,3$}_{(\zeta_3^2)}&0&0\\

\end{array}
$$

$$
\begin{array}{c|cc}
v\backslash u&3A_1&2A_1\\
\hline
111,111&(-3q^{9}-3q^{8}-3q^{6}-3q^{5}-2q^{4}-q^{3}-q^{2}-q-1)\Phi_{2}&(2q^{8}+q^{6}+q^{5}+q^{4}+1)\Phi_{2}^{2}\Phi_{3}\Phi_{6}\\
21,21&(-2q^{2}-1)q^{7}\Phi_{2}^{3}\Phi_{6}&q^{9}\Phi_{2}^{3}\Phi_{3}\Phi_{4}\Phi_{6}\\
3,3&0&0\\
\mbox{$3,3$}_{(\zeta_3)}&0&0\\
\mbox{$3,3$}_{(\zeta_3^2)}&0&0\\

\end{array}
$$

$$
\begin{array}{c|cc}
v\backslash u&A_1&1\\
\hline
111,111&(2q^{10}+q^{9}+q^{8}+q^{7}+2q^{6}+q^{5}+q^{4}+q^{2}+q+1)\Phi_{2}^{3}\Phi_{6}\Phi_{10}&\Phi_{2}^{4}\Phi_{3}\Phi_{4}\Phi_{6}^{2}\Phi_{8}\Phi_{10}\Phi_{12}\Phi_{18}\\
21,21&0&0\\
3,3&0&0\\
\mbox{$3,3$}_{(\zeta_3)}&0&0\\
\mbox{$3,3$}_{(\zeta_3^2)}&0&0\\
\end{array}
$$
\end{footnotesize}

\end{document}